\documentclass[preprint]{elsarticle}       
\usepackage{lipsum}
\makeatletter
\def\ps@pprintTitle{%
 \let\@oddhead\@empty
 \let\@evenhead\@empty
 \def\@oddfoot{}%
 \let\@evenfoot\@oddfoot}
\makeatother
\usepackage{lineno,hyperref}
\usepackage{rotating}
\usepackage{makeidx}
\usepackage{float}
\usepackage{epsfig}
\usepackage{amsmath,amssymb}
\usepackage{enumitem}
\usepackage{multirow}
\usepackage[latin1]{inputenc}
\usepackage{epic,eepic}
\usepackage{overpic}
\usepackage{color}
\usepackage{amsfonts}
\usepackage{graphicx}
\usepackage{epstopdf}
\usepackage{indentfirst}
\usepackage{MnSymbol}
 \usepackage{lscape}
\usepackage{pgf,tikz,pgfplots}
\pgfplotsset{compat=1.15}
\usepackage{mathrsfs}
\usetikzlibrary{arrows}

\usepackage{lipsum}
\usepackage{algorithmic}
\usepackage{amsmath,amssymb, amsthm}
\usepackage{amssymb,amsmath, amsthm,epsfig,multirow}
\usepackage{epic,eepic}
\usepackage{overpic}
\usepackage{color}
\usepackage{amsfonts}
\usepackage{mathrsfs}
\usetikzlibrary{arrows}
\usepackage{mwe} 
\usepackage{subcaption}

\usepackage{multirow}

            {\hfill\penalty10000\raisebox{-.09em}{\large\bf\rm $\blacksquare$}\par\medskip}

\newtheorem{theorem}{Theorem}[section]
\newtheorem{definition}{Definition}[section]

\newtheorem{proposition}[theorem]{Proposition}
\newtheorem{corollary}[theorem]{Corollary}

\begin{document}
\small

\begin{frontmatter}

\title{General adaptive rational interpolation with maximum order close to discontinuities}\tnotetext[label1]{This work was funded by Spanish MINECO project MTM2017-83942-P}

\author[UV]{F. Ar\`andiga}
\ead{arandiga@uv.es}
\author[UV]{Dionisio F. Y\'a\~nez$^*$\footnote{$^*$Corresponding author.}}
\ead{dionisio.yanez@uv.es}
\date{Received: date / Accepted: date}

\address[UV]{Departamento de Matem\'aticas, Facultad de Matem\'aticas. Universidad de Valencia. Valencia, Spain.}




%
%




\begin{abstract}
 Adaptive rational interpolation has been designed in the context of image processing as a new nonlinear technique that avoids the Gibbs phenomenon when we approximate a discontinuous  function. In this work, we present a generalization to the method giving explicit expressions for all the weights for any order of the algorithm. It has a similar behavior to weighted essentially non oscillatory (WENO) technique, however because of the design of the weights in this case is more simple, we propose a new way to construct them obtaining the maximum order near the discontinuities. Some experiments are performed to demonstrate our results and to compare them with standard methods.
\end{abstract}



\end{frontmatter}

\section{Introduction}
\label{sec:intro}

A classical problem in approximation theory is to reconstruct a continuous function from a set of discrete points. These data can be interpreted
as a discretization of a function. Depending on the context where the problem appears different settings can be considered. In the literature, the most common discretizations used and developed are point-value, if the data are the values of the function at a finite set of points (see, e.g. \cite{amat:weno1,arandiga:jcam2019,belda:1,belda:2,arandiga:1,harten:2}) and cell-average; in that case the values are the averages of an integrable function (see, e.g. \cite{arandiga:1,liu:1,harten:2,harten:3}). Others frameworks as hat-average, see \cite{arandiga:2} are employed in vortex methods for the numerical solution of partial differential equations. Fixed the discretization more suitable for the type of problem posed, the data are interpolated, i.e., a continuous function, is designed with the same values, cell-averages or hat-averages at the given set of points, cells or hat-values respectively.

Polynomial interpolatory technique is a well-known procedure that has been satisfactory used in some applications as digital signals or image processing or curves and surfaces
generation through subdivision schemes. Lagrange piecewise interpolation consists on constructing a polynomial with a subset of the values on the given set of points (stencil). However, its accuracy is limited, if the stencil used crosses a discontinuity, Gibbs phenomenon appears and the order of approximation is lost around this point. Thus, if the number of points of the stencil is large to obtain a high order of approximation then the number of regions where a poor approximation is produced increases as it is mentioned in \cite{belda:1}.

In order to avoid discontinuities between the points of the stencil a nonlinear procedure is necessary. In \cite{harten:1}, Harten et al. developed Essentially Non Oscillatory (ENO) technique working on a numerical solution of nonlinear hyperbolic conservation laws. Assuming that the discontinuities are sufficiently separated, ENO method avoids Gibbs phenomenon
picking out the stencil less contaminated by a discontinuity (for more details, see, e.g. \cite{arandiga:1,harten:2} and the references therein).

However, the stability of ENO method is not ensured and some complementary results are necessary in order to use it in some applications. Liu et al. in \cite{liu:1} presented an alternative algorithm based on a nonlinear convex combination of the polynomial constructed with the different possible stencils used in each interpolation of the ENO method. This new procedure, called weighted ENO (WENO), is an improvement stable of the method ENO. There exist several references on the construction of the nonlinear weights (see, e.g. \cite{amatruizshuyanez,belda:1,shu:1}).
In order to obtain the maximum order at the smooth zones, the nonlinear weights have to satisfy among others conditions: to reduce the contribution of the polynomials corresponding to the stencils which cross a discontinuity. One of the problems that this procedure presents is whereas the accuracy in smooth zones is optimal due to the design of the weights, in the zones close to the discontinuities it is reduced because of the usage of the same optimal weights. In \cite{amatruizshuyanez} a new algorithm is designed to improve the order of the approximation close to discontinuities.

In \cite{carrato:1,castagno:1,castagno:2,ramponi:1} Ramponi et al. proposed a new nonlinear algorithm of order two. Ar\`andiga in \cite{arandiga:jcam2019}, using the principal ideas of WENO interpolation, improves this method in order 2 and extends it to order 4. However, the method, as WENO, reproduces the same problem in the accuracy of the interpolant close to discontinuities. In this paper, we generalize this method to any even order. Also, we construct a family of general weights which satisfy the necessary properties to obtain maximum order in all the intervals close to discontinuities if they are sufficiently separated.  We offer explicit formulas for the weights based on new optimal weights depending on the number of intervals not contaminated by the jump discontinuity. This method is easily replicable in cell-average (see \cite{arandiga.yanez:1}) or hat-average settings.

This paper is organized in 6 sections: We review point-value WENO interpolation and explain the different conditions that the weights have to satisfy in order to obtain maximum order interpolants in smooth zones. In Section \ref{rational:interpolation} we sketch the work developed by Ar\`andiga in \cite{arandiga:1} on adaptive rational interpolation for order 2 and 4. Section \ref{sec:newweights2r} is devoted to explain how generalize these results obtaining maximum order close to discontinuities. Also we prove the principal properties and introduce an explicit formula for any order $2r$. In Section \ref{sec:numerexperiments}, we perform some tests where the improves are reflected in comparison with the mentioned above methods and some conclusions are presented. Finally, some conclusions and future work are showed in Section \ref{sec:conclusions}.

\section{Point-value linear interpolation versus WENO method}
\label{sec:weno}

Adaptive rational interpolation designed by Ramponi et al. in \cite{carrato:1,castagno:2,castagno:1,ramponi:1} can be seen as a comparable method to WENO. The construction of the weights in rational interpolation is similar. Also, for the adaptation to order 4 presented by Ar\`andiga in \cite{arandiga:1}, most of the concepts involved in WENO interpolation are used. In order to link general rational interpolation with point-value WENO interpolation, in this section we review WENO using the ideas and the notation presented in \cite{belda:2}. We explain it in comparison with the well-known linear interpolation method.

Let be $i$ the level of the resolution and $X^i=\{x_j^i\}_{j=0}^{N_i}$ a uniform discretization of the interval $[0,1]$, with $x_j^i=j h_i$, $h_i=1/N_i$ and let be $\{f_j^i\}_{j=0}^{N_i}$ the point values of a function $f(x)$, being $f_j^i=f(x_j^i)$. The objective is to obtain an approximation of the values of the function $f$ in a finer grid $X^{i+1}=\{x_j^{i+1}\}_{j=0}^{N_{i+1}},$ with $h_{i+1}=1/N_{i+1}, N_{i+1}=2N_i$.
Firstly, we replicate the values of the function in the even points, then:
\begin{equation}
\hat f_{2j}^{i+1}:= f_{j}^{i},\,\,\, j=0,\dots, N_i,
\end{equation}
and for the even points we will use an interpolatory approach.

Lagrange piecewise linear interpolation consists on approximating the value of the function at the point $x^{i+1}_{2j-1}$, (we denote it as $x^{i}_{j-\frac1 2}$) constructing a polynomial with $2r$ nodes around this point. Hence, we denote the stencil $S_{r-1}^{2r-1}=\{x^i_{j-r},\dots,x^i_{j+r-1}\}$, and  construct the polynomial $p_{r-1}^{2r-1}(x)$ which interpolates on this stencil to define the approximation as:
\begin{equation}\label{eq:linear}
\hat f_{2j-1}^{i+1}=\hat f_{j-\frac 1 2}^{i}:= p_{r-1}^{2r-1}(x^{i}_{j-\frac1 2}) = \sum_{\xi=-r}^{r-1} \beta^{2r-1}_\xi f^i_{j+\xi},\,\,
\end{equation}
with
\begin{equation*}
\beta^{2r-1}_\xi=\frac{(-1)^{2r-1}}{2^{2r-1}}\prod_{j=-r,j\neq \xi}^{r-1} \left(\frac{2j+1}{i-j}\right).
\end{equation*}
Then, it is easy to check that $\beta_{-\xi}^{2r-1} = \beta_{\xi-1}^{2r-1}, \,\xi=1,\dots,r$.

If we suppose that $f$ is
smooth on the region determined by the stencil $S_{r-1}^{2r-1}$, then
\begin{equation}\label{ordertotal}
f_{2j-1}^{i+1}= \hat f_{2j-1}^{i+1} + O(h_i^{2r}).
\end{equation}
However, if a discontinuity crosses the stencil $S_{r-1}^{2r-1}$ then the accuracy is lost if we use a data-independent polynomial. In order to solve this problem, Liu et al. in \cite{liu:1} propose a convex combination of the polynomials $\{p^r_{k}(x)\}_{k=0}^{r-1}$ which interpolate on the points of the sub-stencils  $S^{r}_k=\{x^i_{j+k-r},\dots,x^i_{j+k}\}$, $k=0,\dots,r-1$. Thus, we denote as:
\begin{equation}
q(x)=\sum_{k=0}^{r-1} \omega^j_k p_{k}^r(x), \,\, \text{ with }\,\, \omega^j_k\geq 0,\,\, k=0, \dots, r-1, \,\,\, \text{ and } \,\,\, \sum_{k=0}^{r-1} \omega^j_k=1,
\end{equation}
and the approximation in the point $x^{i}_{j-\frac 1 2}$ would be:
\begin{equation*}
\hat f_{j-\frac 1 2}^{i}:=q(x^{i}_{j-\frac 1 2})=\sum_{k=0}^{r-1} \omega^j_k p_{k}^r(x^{i}_{j-\frac 1 2}).
\end{equation*}
Note that $\{S^r_k\}_{k=0}^{r-1}$ are all the possible stencils that contain the points $x^i_{j-1}$ and $x^i_j$.

Since the level of resolution $i$ and the point $j$ are fixed, in the rest of the paper we denote as $h$, $x_{j}$, $x_{j-\frac 1 2}$, $f_j$, $\hat f_{j-\frac 1 2}$, $\omega_k$ instead of $h_i$, $x^i_{j}$, $x^i_{j-\frac 1 2}$, $f^i_j$, $\hat f_{j-\frac 1 2}^{i}$, $\omega^j_k$.

The construction of the nonlinear weights is crucial to obtain the maximum order of accuracy, $2r$, in the smooth zones, Eq. \eqref{ordertotal}. Henceforth, we can define the optimal weights (see \cite{belda:2}) as the constants $C^{2r-1}_{r-1,k}> 0$, $k=0,\dots,r-1$, such that:
 \begin{equation}
p_{r-1}^{2r-1}(x_{j-\frac 1 2})=\sum_{k=0}^{r-1} C^{2r-1}_{r-1,k} p_{k}^r(x_{j-\frac 1 2}), \,\, \text{ and } \,\, \sum_{k=0}^{r-1} C^{2r-1}_{r-1,k}=1.
\end{equation}
Also, in \cite{belda:2}, an explicit formula is provided proving the following proposition.
\begin{proposition}{(Ar\`andiga et al. \cite{belda:2})}\label{Propox1}
The optimal weights are given by
\begin{equation}
C^{2r-1}_{r-1,k}=\frac{1}{2^{2r-1}} {2r \choose 2k+1}, \quad k=0,\dots,r-1.
\end{equation}
\end{proposition}
Hence, if $f$ is smooth in the interval $[x_{j-r},x_{j+r-1}]$ and the nonlinear weights satisfy the following condition:
\begin{equation}\label{prop1}
\omega^j_k=C^{2r-1}_{r-1,k}+O(h^m), \,\, k=0,\dots,r-1,
\end{equation}
with $m\leq r-1$ then, it can be proved (see \cite{belda:2,liu:1}) that
\begin{equation}
f(x_{j-\frac 1 2})-q(x_{j-\frac 1 2})=O(h^{r+m+1}),
\end{equation}
then if $m=r-1$ the maximum order is obtained.

Typically, see \cite{liu:1,belda:2}, the nonlinear weights are defined as:
 \begin{equation*}
                  \omega^j_k=\frac{\alpha^j_k}{\sum_{\xi=0}^{r-1}\alpha_\xi^j}, \qquad k=0,\dots, r-1, \,\,\text{with } \,\, \alpha_k^j=\frac{C^{2r-1}_{r-1,k}}{(\varepsilon_h+I_k^j)^t}
\end{equation*}
being $t$ an integer that ensures maximum order of accuracy near the discontinuities and where $\varepsilon_h$ is defined in \cite{liu:1} as $\varepsilon=10^{-6}$. This is a constant introduced to avoid the null denominators. However, in \cite{belda:2}, it is proved that the choice of this constant is crucial to obtain the maximum order when $h$ is small. The authors analyze different examples and determine that it depends on $h$ being $\varepsilon=h^2$ the suitable value to gain an accuracy approximation.
The values $I^j_k$ are the indicators of the smoothness of the function $f$ in the stencil $S_{k}^r$. Thus, if $f$ is smooth in $S_{k}^r$ then $I^j_k=O(h^2)$ and if there exists a discontinuity in this stencil $I^j_k=O(1)$. There are different ways to design these indicators. In Jiang and Shu, \cite{liu:1}, an indicator of smoothness is defined in cell-average context as:
\begin{equation*}
L^j_k=\sum_{l=1}^r \int_{x_{j-\frac12}}^{x_{j+\frac12}} h^{2l-1}((p^r_{k})^{(l)}(x))^2 dx.
\end{equation*}
In \cite{amat:weno1} Amat et al. change this expression in order to detect corner discontinuities. In this paper, we will use one of the three families of smoothness indicators constructed in \cite{belda:2} defined as:
\begin{equation}\label{equationsuavidad}
I^{r,j}_k=\sum_{l=1}^r \int_{x_{j-1}}^{x_{j}} h^{2l-1}((p^r_{k})^{(l)}(x))^2 dx.
\end{equation}

In order to make the paper self-contained we show the explicit expressions for $r = 3$:
\begin{equation*}
I^{3,j}_k=(\delta^k_1)^2-2\delta^k_1\delta^k_2+2\delta^k_1\delta^k_3+\frac{16}{3}(\delta^k_2)^2-15\delta^k_2\delta^k_3+\frac{249}{5}(\delta^k_3)^2,     \quad k=0,1,2,
\end{equation*}
being the coefficients for $p_{0}^3$:
\begin{equation*}
\begin{split}
\delta^0_1&=-\frac13f_{j-3}+\frac32f_{j-2}-3f_{j-1}+\frac{11}{6}f_{j},\\
\delta^0_2&=-\frac12f_{j-3}+2f_{j-2}-\frac52f_{j-1}+f_{j},\\
\delta^0_3&=-\frac16f_{j-3}+\frac12f_{j-2}-\frac12f_{j-1}+\frac{1}{6}f_{j},\\
\end{split}
\end{equation*}
for $p_{1}^3$:
\begin{equation*}
\begin{split}
\delta^1_1&=\frac16f_{j-2}-f_{j-1}+\frac12f_{j}+\frac{1}{3}f_{j+1},\\
\delta^1_2&=\frac12f_{j-1}-f_{j}+\frac12f_{j+1},\\
\delta^1_3&=-\frac16f_{j-2}+\frac12f_{j-1}-\frac12f_{j}+\frac{1}{6}f_{j+1},\\
\end{split}
\end{equation*}
and for $p_{2}^3$:
\begin{equation*}
\begin{split}
\delta^2_1&=-\frac13f_{j-1}-\frac12f_{j}+f_{j+1}-\frac{1}{6}f_{j+2},\\
\delta^2_2&=\frac12f_{j-1}-f_{j}+\frac12f_{j+1},\\
\delta^2_3&=-\frac16f_{j-1}+\frac12f_{j}-\frac12f_{j+1}+\frac{1}{6}f_{j+2},\\
\end{split}
\end{equation*}

An for $r=4$,

\begin{equation*}
\begin{split}
I^{4,j}_k=&\frac{1748}{35} (\delta^k_1)^2+\frac{83 }{15}(\delta^k_2)^2+\frac{4}{3}(\delta^k_3)^2+(\delta^k_4)^2-\frac{46}{3}\delta^k_1 \delta^k_2\\
&+\frac{12}{5}\delta^k_1 \delta^k_3-\frac{1}{2}\delta^k_1 \delta^k_4-\frac{5}{2}\delta^k_2 \delta^k_3+\frac{2}{3}\delta^k_2 \delta^k_4-\delta^k_3 \delta^k_4,\quad k=0,1,2,3,
\end{split}
\end{equation*}
being the coefficients for $p_{0}^4$:
\begin{equation*}
\begin{split}
\delta^0_1&=\frac16(f_{j}-4 f_{j-1}+6 f_{j-2}-4 f_{j-3}+f_{j-4}),\\
\delta^0_2&=\frac14(5 f_{j}-18 f_{j-1}+24 f_{j-2}-14 f_{j-3}+3 f_{j-4}),\\
\delta^0_3&=\frac{1}{12}(35 f_{j}-104 f_{j-1}+114 f_{j-2}-56 f_{j-3}+11 f_{j-4}),\\
\delta^0_4&=\frac{1}{12}(25 f_{j}-48 f_{j-1}+36 f_{j-2}-16 f_{j-3}+3 f_{j-4}),\\
\end{split}
\end{equation*}
for $p_{1}^4$:
\begin{equation*}
\begin{split}
\delta^1_1&=\frac{1}{6} (f_{j+1}-4 f_{j}+6 f_{j-1}-4 f_{j-2}+f_{j-3}),\\
\delta^1_2&=\frac{1}{4} (3f_{j+1}-10 f_{j}+12 f_{j-1}-6 f_{j-2}+f_{j-3}),\\
\delta^1_3&=\frac{1}{12} (11f_{j+1}-20 f_{j}+6 f_{j-1}+4 f_{j-2}-f_{j-3}),\\
\delta^1_4&=\frac{1}{12} (3f_{j+1}+10 f_{j}-18 f_{j-1}+6 f_{j-2}-f_{j-3}),\\
\end{split}
\end{equation*}
for $p_{2}^4$:
\begin{equation*}
\begin{split}
\delta^2_1&=\frac{1}{6} (f_{j+2}-4 f_{j+1}+6 f_{j}-4 f_{j-1}+f_{j-2}),\\
\delta^2_2&=\frac{1}{4} (f_{j+2}-2 f_{j+1}+2 f_{j-1}-f_{j-2}),\\
\delta^2_3&=\frac{1}{12} (-f_{j+2}+16 f_{j+1}-30 f_{j}+16 f_{j-1}-f_{j-2}),\\
\delta^2_4&=\frac{1}{12} (-f_{j+2}+8 f_{j+1}-8 f_{j-1}+f_{j-2}),\\
\end{split}
\end{equation*}
and for $p_{3}^4$:
\begin{equation*}
\begin{split}
\delta^3_1&=\frac{1}{6} (f_{j+3}-4 f_{j+2}+6 f_{j+1}-4 f_{j}+f_{j-1}),\\
\delta^3_2&=\frac{1}{4} (-f_{j+3}+6 f_{j+2}-12 f_{j+1}+10 f_{j}-3 f_{j-1}),\\
\delta^3_3&=\frac{1}{12} (-f_{j+3}+4 f_{j+2}+6 f_{j+1}-20 f_{j}+11 f_{j-1}),\\
\delta^3_4&=\frac{1}{12} (f_{j+3}-6 f_{j+2}+18 f_{j+1}-10 f_{j}-3 f_{j-1}).\\
\end{split}
\end{equation*}

\section{Point-value adaptive rational interpolation}\label{rational:interpolation}
In this section we introduce the rational interpolation proposed by S. Carrato and G. Ramponi \cite{carrato:1} in the point-value setting. This is a second order  nonlinear interpolatory technique, \cite{arandiga:jcam2019}. It consists on obtaining an approximation of the value $f_{j-\frac 1 2}$ by a weighted average between $f_{j-1}$ and $f_j$, i.e.
$$\hat f_{j-\frac 1 2}=\omega_{0} f_{j-1} +\omega_{1} f_{j}, \,\, \text{ with } \,\,\omega_{0}+\omega_1=1.$$
In this case, the polynomials chosen are of degree 0 since $f_j=p^0_j(x_j)$ and $f_{j-1}=p^0_{j-1}(x_{j-1})$. This is not exactly a WENO interpolation because of the sub-stencils do not contain the points $x_{j-1}$ and $x_j$ but the construction is similar.  In \cite{carrato:1}, two families of weights are proposed. We denote them as:
\begin{equation} \label{eq:racional-pesos2}
\begin{split}
\omega_{1,0} &=  \frac{1+\alpha(f_{j-1}-f_{j+1})^2}{2+\alpha((f_{j-2}-f_{j})^2 + (f_{j-1}-f_{j+1})^2)},\\
\omega_{1,1} &=\frac{1+\alpha(f_{j-2}-f_{j})^2}{2+\alpha((f_{j-2}-f_{j})^2 + (f_{j-1}-f_{j+1})^2)},  \\
\end{split}
\end{equation}
and
\begin{equation} \label{eq:racional-pesos2.2}
\begin{split}
\omega_{2,0} & =  \frac{1+\alpha((f_{j-1}-f_{j+1})^2 + (f_i-f_{j+1})^2)}{2+\alpha(\sum_{s=0}^1((f_{j-1}-f_{j+1-3s})^2+(f_{j}-f_{j+1-3s})^2))},\\
\omega_{2,1} & =  \frac{1+\alpha((f_{j-1}-f_{j-2})^2+(f_{j}-f_{j-2})^2)}{2+\alpha(\sum_{s=0}^1((f_{j-1}-f_{j+1-3s})^2+(f_{j}-f_{j+1-3s})^2))},
\end{split}
\end{equation}
where $\alpha$ is a  parameter. If $f$ is smooth in  $[x_{j-2}, x_{j+1}]$  then the following property is true (see \cite{arandiga:jcam2019}):
\begin{equation}
\hat f_{j-\frac 1 2}=\omega_{s,0} f_{j-1} +\omega_{s,1} f_{j}=f(x_{j-\frac1 2})+O(h^2),\,\, s=1,2.
\end{equation}

Furthermore, Ar\`andiga presents in \cite{arandiga:jcam2019}, rational interpolation of order four ($r=2$) as a modification in the design of the weights of WENO method. Thus, he calculates an approximation:
\begin{equation}
\hat f_{j-\frac 1 2}= \omega_{3,0} p^2_0(x_{j-\frac 1 2}) + \omega_{3,1} p^2_{1}(x_{j-\frac 1 2}),
\end{equation}
with the weights defined as:
\begin{equation*}
\omega_{3,k}=\frac{\alpha_{3,k}}{\alpha_{3,0}+\alpha_{3,1}}, \quad \alpha_{3,k}=\frac{1/2}{(\epsilon+I^{2,j}_k)^t}, \quad k=0,1,
\end{equation*}
where $I^{2,j}_k$ are the smoothness indicator defined in Eq. \eqref{equationsuavidad}. Based in these weights, in \cite{arandiga:jcam2019}, the following nonlinear weights are proposed:
\begin{equation*}
\omega_{4,k}=\frac{\alpha_{4,k}}{\alpha_{4,0}+\alpha_{4,1}}, \quad \alpha_{4,k}=1+h^{-2t}(I^{2,j}_k)^t, \quad k=0,1.
\end{equation*}

Using these ideas we construct a general way to obtain rational interpolation for any order $2r$. In this generalization, we design the new weights in order to satisfy the condition showed in Eq. \eqref{prop1}. Also, we adapt them to the different optimal weights depending on the situation of the discontinuity. To explain the procedure with more clarity, let us start with the example $r=3$ and in Section \ref{sec:newweights2r}, we perform the general case.

\subsection{Design of new nonlinear weights: The case $r=3$}
\label{sec:newweights}
In this section, we construct a family of weights that tend to the different optimal weights depending on the interval which contains the discontinuity. We explain the design of the weights for the case $r=3$, i.e. using 6 points.

Let's start with the stencil $S^5_{2}=\{x_{j-3},x_{j-2},x_{j-1},x_{j},x_{j+1},x_{j+2}\}$, and sub-stencils $S^3_{0}=\{x_{j-3},x_{j-2},x_{j-1},x_{j}\}$, $S^3_{1}=\{x_{j-2},x_{j-1},x_{j},x_{j+1}\}$, $S^3_{2}=\{x_{j-1},x_{j},x_{j+1},x_{j+2}\}$,
$S^4_{1}=\{x_{j-3},x_{j-2},x_{j-1},x_{j},x_{j+1}\}$,
$S^4_{2}=\{x_{j-2},x_{j-1},x_{j},x_{j+1},x_{j+2}\}$ to construct the interpolatory polynomials $p^5_{2}(x)$, $p^3_{0}(x)$, $p^3_{1}(x)$, $p^3_{2}(x)$, $p^4_{1}(x), p^4_{2}(x)$
respectively.

\begin{figure} \begin{center}\begin{tikzpicture}[x=8.5cm,y=8.5cm]
\draw [line width=0.8pt] (0.1,-0.005)-- (0.1,0.005);
\draw [line width=0.8pt] (0.2,-0.005)-- (0.2,0.005);
\draw [line width=0.8pt] (0.3,-0.005)-- (0.3,0.005);
\draw [line width=0.8pt] (0.4,-0.005)-- (0.4,0.005);
\draw [line width=0.8pt] (0.5,-0.005)-- (0.5,0.005);
\draw [line width=0.8pt] (0.6,-0.005)-- (0.6,0.005);
\draw [line width=0.8pt] (0.7,-0.005)-- (0.7,0.005);
\draw [line width=0.8pt] (0.8,-0.005)-- (0.8,0.005);
\draw [line width=0.8pt] (0.9,-0.005)-- (0.9,0.005);
\draw (0.060203299296807844,-0.01) node[anchor=north west] {$x_{j-4}$};
\draw (0.15958653082592375 ,-0.01) node[anchor=north west] {$x_{j-3}$};
\draw (0.2654512774547646  ,-0.01) node[anchor=north west] {$x_{j-2}$};
\draw (0.4793412757456879  ,-0.01) node[anchor=north west] {$x_j$};
\draw (0.559279961975629   ,-0.01) node[anchor=north west] {$x_{j+1}$};
\draw (0.7861329904660023  ,-0.01) node[anchor=north west] {$x_{j+3}$};
\draw (0.8876767270283598  ,-0.01) node[anchor=north west] {$x_{j+4}$};
\draw (0.3648345089838805  ,-0.01) node[anchor=north west] {$x_{j-1}$};
\draw (0.6673052136377114  ,-0.01) node[anchor=north west] {$x_{j+2}$};
\draw [line width=1pt] (-0.05430346746499953,0)-- (1.1469373310173576,0);
\draw [->,line width=1pt] (0.4,0.07) -- (0.2,0.07);
\draw [->,line width=1pt] (0.4,0.07) -- (0.5,0.07);
\draw [->,line width=1pt] (0.4,0.14) -- (0.6,0.14);
\draw [->,line width=1pt] (0.4,0.14) -- (0.3,0.14);
\draw [->,line width=1pt] (0.5,0.21) -- (0.7,0.21);
\draw [->,line width=1pt] (0.5,0.21) -- (0.4,0.21);
\draw [->,line width=1pt] (0.4,-0.07) -- (0.2,-0.07);
\draw [->,line width=1pt] (0.4,-0.07) -- (0.6,-0.07);
\draw [->,line width=1pt] (0.4,-0.14) -- (0.3,-0.14);
\draw [->,line width=1pt] (0.4,-0.14) -- (0.7,-0.14);
\draw (0.516069861310796,0.28) node[anchor=north west] {$S^3_2$};
\draw (0.42964965998112997,0.21) node[anchor=north west] {$S^3_1$};
\draw (0.3302664284520141,0.14) node[anchor=north west] {$S^3_0$};
\draw (0.37779753918333037,-0.07) node[anchor=north west] {$S^4_1$};
\draw (0.46421774051299636,-0.14) node[anchor=north west] {$S^4_2$};
\draw (-0.013253871833408256,-0.01) node[anchor=north west] {$\dots$};
\draw (0.9935414736572006,-0.01) node[anchor=north west] {$\dots$};
\begin{scriptsize}
\draw  (0.2,0) circle (2pt);
\draw  (0.3,0) circle (2pt);
\draw  (0.4,0) circle (2pt);
\draw  (0.5,0) circle (2pt);
\draw  (0.6,0) circle (2pt);
\draw  (0.7,0) circle (2pt);
\end{scriptsize}
\end{tikzpicture}
\end{center}
\caption{Stencils and possible sub-stencils used when $r=3$ and any discontinuity does not cross the stencil $S^{5}_2$ the optimal weights are the values $C^5_{k}$ plus values of order 2 or higher, with $k=0,1,2$. }
\label{fig1}
\end{figure}
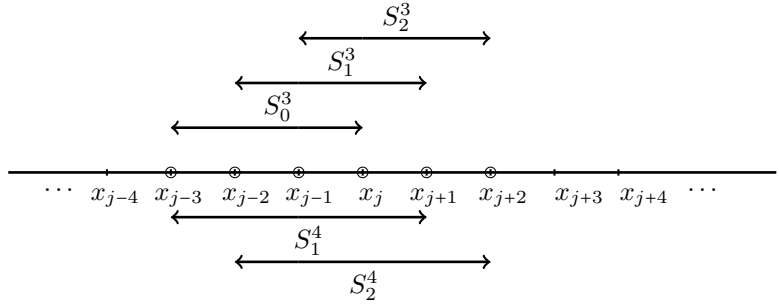

Then, we can express all the approximations using the above polynomials as a linear convex combination of the approximation obtained using polynomials of degree $3$. For this, we define the following concept that we use throughout all the paper.
\begin{definition}\label{def:optimalweights}
With the notation used during all the paper. Let be $r\leq a \leq 2r-1$ and $a-r\leq b \leq r-1$ then the $(a,b)$-optimal weights are the values $C^{a}_{b,k}$, $k=0,\dots,a-r$, that satisfy:
\begin{equation*}
p_{b}^{a}(x_{j-\frac 1 2})=\sum_{k=0}^{a-r}C^{a}_{b,k} p^r_{k}(x_{j-\frac 1 2}), \quad\text{with}\quad \sum_{k=0}^{a-r}C^{a}_{b,k}=1.
\end{equation*}
\end{definition}
Note that $(2r-1,r-1)$-optimal weights are the classical optimal weights defined in Prop. \ref{Propox1}. In the particular case of $r=3$,  we get:
\begin{equation*}
p_{2}^{5}(x_{j-\frac 1 2})=C^{5}_{2,0} p^3_{0}(x_{j-\frac 1 2})+ C^{5}_{2,1} p^3_{1}(x_{j-\frac 1 2})+C^{5}_{2,2} p^3_{2}(x_{j-\frac 1 2}),
\end{equation*}
with $C^{5}_{2,0}=C^5_{2,2}=\frac{3}{16}$ and  $C^{5}_{2,1}=\frac{5}{8}$ by Prop. \ref{Propox1}. Again, it is easy to calculate that
\begin{equation*}
\begin{split}
& p_{1}^{4}(x_{j-\frac 1 2})=C^{4}_{1,0} p^3_0(x_{j-\frac 1 2})+ C^{4}_{1,1} p^3_{1}(x_{j-\frac 1 2})\\
& p_{2}^{4}(x_{j-\frac 1 2})=C^{4}_{2,1} p^3_{1}(x_{j-\frac 1 2})+ C^{4}_{2,2} p^3_{2}(x_{j-\frac 1 2})\\
\end{split}
\end{equation*}
being $C^{4}_{1,0}=C^{4}_{2,2}=\frac{3}{8}$ and  $C^{4}_{1,1}=C^{4}_{2,1}=\frac{5}{8}$, and obviously  $C^{3}_{k,k}=1$, $k=0,1,2$.
In Fig. \ref{fig1} we sketch these sub-stencils for $r=3$. Depending on the position of the discontinuity the nonlinear weights will have a negligible influence or will be the $(a,b)$-optimal weights plus elements of order $O(h^2)$ or higher. Thus, we define:
\begin{equation*}
\omega_{k}=\frac{\alpha_{k}}{\alpha_{0}+\alpha_{1}+\alpha_{2}}, \quad k=0,1,2,
\end{equation*}
with
\begin{equation}
\begin{split}
\alpha_0&=C^{5}_{3,0}+h^{-t}\big(C^4_{1,0} J_{2}+C^{3}_{0,0} J_{1}\big)=\frac{3}{16}+h^{-t}\left(\frac{3}{8}J_{2}+J_{1}\right),\\
\alpha_1&=C^{5}_{3,1}+h^{-t}\big(C^{4}_{1,1} J_{2}+C^{4}_{2,1}J_{-2}\big)=\frac{5}{8}+h^{-t}\left(\frac{5}{8}J_{2}+\frac{5}{8}J_{-2}\right),\\
\alpha_2&=C^{5}_{3,2}+h^{-t}\big(C^{4}_{2,2} J_{-2}+C^{3}_{2,2} J_{-1}\big)=\frac{3}{16}+h^{-t}\left(\frac{3}{8} J_{-2}+ J_{-1}\right),\\
\end{split}
\end{equation}
being
\begin{equation}
J_{\xi}=|f_{j+\xi}-f_{j+\xi-1}|^{2t},\quad \xi=-2,\dots,2.
\end{equation}
In order to ensure the maximum order of accuracy we will take $t=2r-1$, in this case $t=5$. We can rewrite the above expressions as:
\begin{equation}
\left(
         \begin{array}{c}
 \alpha_0 \\
 \alpha_1 \\
 \alpha_2 \\
\end{array}
\right)=
       \left(
         \begin{array}{c}
              3/16\\
 5/8 \\
 3/16 \\
\end{array}
\right)+
h^{-t}\left(
  \begin{array}{cccc}
    3/8 & 1 & 0   & 0 \\
    5/8 & 0 & 0   & 5/8  \\
    0   & 0 & 1   & 3/8 \\
  \end{array}
\right)    \left(
         \begin{array}{c}
 J_{2} \\
 J_{1} \\
 J_{-1} \\
 J_{-2} \\
\end{array}
\right).
\end{equation}

With these definitions we analyze the following different situations respect to the discontinuity:
\begin{enumerate}
\item If $f$ is smooth in the interval $[x_{j-3},x_{j+2}]$ then the nonlinear weights are the optimal weights defined in Eq. \eqref{prop1} plus elements of order $h^5$, i.e.,
\begin{equation}\label{prop:pesos60}
\omega_{k}=C^{5}_{3,k} + O(h^t), \quad k=0,1,2,
\end{equation}
this is easy to prove from $J_{\xi}=O(h^{2t})$, $\xi=-2,\dots,2$, (see Prop. \ref{prop:2} below).
\item  If $f$ has a discontinuity at the interval $[x_{j-3},x_{j-2}]$ and it is smooth in the rest of the intervals then the stencil $S^4_{2}$ does not contain any discontinuity. Hence,
the nonlinear weights will be $(4,2)$-optimal weights plus elements of order 5. In particular, in Section \ref{sec:newweights2r} we prove that
\begin{equation}\label{prop:pesos63}
\omega_{k}=\left\{
                \begin{array}{ll}
                   C^{4}_{2,k} + O(h^t), & k=1,2; \\
                   O(h^t), & k=0.
                \end{array}
              \right.
\end{equation}
\item  If $f$ has a discontinuity at the interval $[x_{j+1},x_{j+2}]$  and it is smooth at the stencil $S^4_1$, the situation is similar to the case 2. Thus, we prove that the nonlinear weights are $(4,1)$-optimal weights plus elements of order 5.
    \item  Finally, if $f$ has a discontinuity at the interval $[x_{j},x_{j+1}]$ or $[x_{j-2},x_{j-1}]$ then the maximum order obtained is $r+1=4$, we get
    \begin{equation}\label{prop:pesos64}
\omega_{k}=\left\{
                \begin{array}{ll}
                   C^{3}_{2,k} + O(h^t)=1+O(h^t), & k=2; \\
                   O(h^t), & k=0,1.
                \end{array}
              \right.
\end{equation}
\end{enumerate}

With these results we can prove the following theorem.

\begin{theorem}
If $f$ is smooth in $[x_{j-1},x_j]$, then
\begin{equation}
\hat{f}_{j-\frac12}-{f}_{j-\frac12}=\left\{
                                                  \begin{array}{ll}
                                                    O(h^6), & \hbox{if $f$ \text{ is smooth in } $[x_{j-3},x_{j+2}]$;} \\
                                                    O(h^5), & \hbox{if $f$ \text{ has a discontinuity at } $[x_{j-3},x_{j-2}] \text{ or } [x_{j+1},x_{j+2}]$;} \\
                                                    O(h^4), & \hbox{if $f$ \text{ has a discontinuity at } $[x_{j-2},x_{j-1}] \text{ or } [x_{j},x_{j+1}]$. }
                                                  \end{array}
                                                \right.
\end{equation}
\end{theorem}

\begin{proof}
We suppose that $f$ is smooth in $[x_{j-3},x_{j+2}]$, then
 \begin{equation*}
    p^5_{3}(x_{j-\frac12}) = f_{j-\frac12}+O(h^6),\quad p^3_{k}(x_{j-\frac1 2}) = f_{j-\frac12}+O(h^4), \quad k=0,1,2,
  \end{equation*}
and using Eq. \eqref{prop:pesos60} we can obtain  the order of accuracy:
\begin{eqnarray*}
\sum_{k=0}^2 \omega_{k}p^3_{k}(x_{j-\frac12})-p^5_{3}(x_{j-\frac12})
&&=\sum_{k=0}^2\omega_{k}p^3_{k}(x_{j-\frac12})
-\sum_{k=0}^2C^{j+3,5}_{k}p^3_{k}(x_{j-\frac12})\\
&&=\sum_{k=0}^2(\omega_k-C^{5}_{3,k})p^3_{k}(x_{j-\frac12})=\sum_{k=0}^2 (\omega_{k}-C^{5}_{3,k})(p^3_{k}(x_{j-\frac12})- f_{i-\frac12})\\
&&=O(h^t)\cdot O(h^4)=O(h^{t+4}).
\end{eqnarray*}

Secondly we suppose that $f$ has a discontinuity at $[x_{j+1},x_{j+2}]$ (analogously for $[x_{j-3},x_{j-2}]$). Since:
 \begin{equation*}
 p^4_{1}(x_{j-\frac1 2}) = f_{j-\frac12}+O(h^5),
  \end{equation*}
and, by Eq. \eqref{prop:pesos63}, we get
\begin{equation*}
\begin{split}
\sum_{k=0}^2 \omega_{k}p^3_{k}(x_{j-\frac12})-f_{j-\frac1 2}=&\sum_{k=0}^2 \omega_{k}p^3_{k}(x_{j-\frac12})-p^4_{1}(x_{j-\frac12})+(f_{j-\frac1 2} - p^4_{1}(x_{j-\frac12}))\\
=&\sum_{k=0}^2\omega_{k}p^3_{k}(x_{j-\frac12})
-\sum_{k=0}^1C^{4}_{1,k}p^3_{k}(x_{j-\frac12})+(f_{j-\frac1 2} - p^4_{1}(x_{j-\frac12}))\\
=&\sum_{k=0}^1(\omega_{k}-C^4_{1,k})p^3_{k}(x_{j-\frac12})
+\omega_{2}p^3_{2}(x_{j-\frac12})+(f_{j-\frac1 2} - p^4_{1}(x_{j-\frac12}))\\
=&\sum_{k=0}^1(\omega_{k}-C^{4}_{1,k})(p^3_{k}(x_{j-\frac12})-f_{j-\frac1 2})
+\omega_{2}(p^3_{2}(x_{j-\frac12})-f_{j-\frac1 2})\\
&+(f_{j-\frac1 2} - p^4_{1}(x_{j-\frac12}))\\
=&O(h^t)O(h^4)+O(h^t)O(1)+O(h^5)=O(h^5).
\end{split}
\end{equation*}

Finally, if $f$ has a discontinuity at $[x_{j},x_{j+1}]$ (analogously for $[x_{j-2},x_{j-1}]$), by Eq. \eqref{prop:pesos64}, we obtain
\begin{equation*}
\begin{split}
\sum_{k=0}^2 \omega_{k}p^3_{k}(x_{j-\frac12})-f_{j-\frac1 2}&=\omega_{0}(p^3_{k}(x_{j-\frac12})-f_{j-\frac1 2})+
\sum_{k=1}^2 \omega_{k}(p^3_{k}(x_{j-\frac12})-f_{j-\frac1 2})\\
&=O(1)O(h^4)+O(h^t)O(1)=O(h^4).
\end{split}
\end{equation*}
\end{proof}

Following the scheme presented for the case $r=3$, we generalize for any order $2r$.

\section{Design of new nonlinear weights: The case $2r$}
\label{sec:newweights2r}
 For the case $2r$, we consider the stencil $S^{2r-1}_{r-1}$ and $\{S^{l+r-1}_{b}, l-1\leq b\leq r-1\}_{l=1}^{r-1}$ to perform the interpolatory polynomials $p^{2r-1}_{r-1}$ and $p^{r+l-1}_{b}$ respectively. We give an explicit form for determined $(a,b)$-optimal
weights proving the following theorem.

\begin{theorem}\label{teoremaweightsr}
With the notation used in this section, let be $1\leq l\leq r$, then
the $(r+l-1,b)$-optimal weights, defined in Def. \ref{def:optimalweights}, with $b=l-1,l$ are given by:
\begin{equation}
\label{eq1}
\begin{split}
C^{r+l-1}_{l-1,k}&=\frac{1}{2^{2l}}\left(\frac{r+l}{l}\right)\left(\frac{l-k}{r-k}\right){l \choose k}{r \choose k}^{-1}{2l \choose l}{l+r \choose r}^{-1}{2r \choose 2k+1}, \quad 0\leq k \leq l,\\
C^{r+l-1}_{l,k}&=\frac{1}{2^{2l}}\left(\frac{2l+1}{2r-1}\right)\left(\frac{r+l}{l}\right)\left(\frac{k}{r-k}\right){l \choose k}{r \choose k}^{-1}{2l \choose l}{l+r \choose r}^{-1}{2r \choose 2k+1}, \,\, 0\leq k \leq l.
\end{split}
\end{equation}
\end{theorem}

\begin{proof}
We use an induction process on $l$. For $l=1$, it is evident. We start $l=2$, if we substitute it in Eq. (\ref{eq1}) we obtain that
\begin{equation*}
\begin{array}{lll}
C^{r+1}_{1,0}=\frac{3}{2(r+1)},& C^{r+1}_{1,1}=\frac{2r-1}{2(r+1)},& C^{r+1}_{1,2}=0,\\
C^{r+1}_{2,0}=0,& C^{r+1}_{2,1}=\frac{5}{2(r+1)}, & C^{r+1}_{2,2}=\frac{2r-3}{2(r+1)},
\end{array}
\end{equation*}
and by Aitken interpolation we get, let be $\xi=1,2$, then
\begin{equation}
\begin{split}
p_{\xi}^{r+1}(x_{j-\frac 1 2})=&\left(\frac{x_{j+\xi}-x_{j-\frac 1 2}}{x_{j+\xi}-x_{j-r+\xi-1}}\right)p_{\xi-1}^{r}(x_{j-\frac 1 2})+\left(\frac{x_{j-\frac 1 2}-x_{j-r}}{x_{j+\xi}-x_{j-r+\xi-1}}\right)p_{\xi}^{r}(x_{j-\frac 1 2})\\
=&\left(\frac{(j+ \xi+1)h-(j-\frac 1 2)h}{(j+\xi)h-(j-r+\xi-1)h}\right)p_{\xi-1}^{r}(x_{j-\frac 1 2})\\
&+\left(\frac{(j-\frac 1 2)h-(j-r+\xi-1)h}{(j+\xi)h-(j-r+\xi-1)h}\right)p_{\xi}^{r}(x_{j-\frac 1 2})\\
=&\left(\frac{2\xi+1}{2(r+1)}\right)p_{\xi-1}^{r}(x_{j-\frac 1 2})+\left(\frac{2(r-\xi)+1}{2(r+1)}\right)p_{\xi}^{r}(x_{j-\frac 1 2})\\
=&C^{r+1}_{\xi,\xi-1}p_{\xi-1}^{r}(x_{j-\frac 1 2})+C^{r+1}_{\xi,\xi} p_{\xi}^{r}(x_{j-\frac 1 2}).\\
\end{split}
\end{equation}

We suppose that Eq. (\ref{eq1}) is true with $2\leq l \leq r-1$, then by induction hypothesis and by Prop. \ref{Propox1} we have that if $l=r-1$, and $0\leq k \leq r$ then
 \begin{equation}\label{equationaitken}
\begin{split}
C^{2r-2}_{r-2,k}&=\frac{1}{2^{2(r-1)}}\left(\frac{2r-1}{r-1}\right)\left(\frac{r-1-k}{r-k}\right){r-1 \choose k}{r \choose k}^{-1}{2r-2 \choose r-1}{2r-1 \choose r}^{-1}{2r \choose 2k+1}\\
&=\frac{2(r-1-k)}{2^{2r-1}(r-1)}{2r \choose 2k+1}=\frac{2(r-1-k)}{r-1}C^{2r-1}_{r-1,k},\\
C^{2r-2}_{r-1,k}&=\frac{1}{2^{2(r-1)}}\left(\frac{2r-1}{r-1}\right)\left(\frac{k}{r-k}\right){r-1 \choose k}{r \choose k}^{-1}{2r-2 \choose r-1}{2r-1 \choose r}^{-1}{2r \choose 2k+1}\\
&=\frac{2k}{2^{2r-1}(r-1)}{2r \choose 2k+1}=\frac{2k}{r-1}C^{2r-1}_{r-1,k}.\\
\end{split}
\end{equation}
Using again Aitken interpolation  we get
\begin{equation}
\begin{split}
p_{r-1}^{2r-1}(x_{j-\frac 1 2})&=\left(\frac{x_{j+r-1}-x_{j-\frac 1 2}}{x_{j+r-1}-x_{j-r}}\right)p_{r-2}^{2r-2}(x_{j-\frac 1 2})+\left(\frac{x_{j-\frac 1 2}-x_{j-r}}{x_{j+r-1}-x_{j-r}}\right)p_{r-1}^{2r-2}(x_{j-\frac 1 2})\\
&=\frac{1}{2}p_{r-2}^{2r-2}(x_{j-\frac 1 2})+\frac{1}{2}p_{r-1}^{2r-2}(x_{j-\frac 1 2}),\\
\end{split}
\end{equation}
thus by Def. \ref{def:optimalweights}, if we substitute the values $C^{2r-2}_{r-2,k}$ and $C^{2r-2}_{r-1,k}$ obtained in Eq. \eqref{equationaitken}, we have
\begin{equation}
\begin{split}
p_{r-1}^{2r-1}(x_{j-\frac 1 2})&=\frac{1}{2}\sum_{k=0}^{r-1}C^{2r-2}_{r-2,k} p^r_{k}(x_{j-\frac 1 2})+\frac{1}{2}\sum_{k=0}^{r-1}C^{2r-2}_{r-1,k} p^r_{k}(x_{j-\frac 1 2}) \\
&=\sum_{k=0}^{r-1}C^{2r-1}_{r-1,k} p^r_{k}(x_{j-\frac 1 2}). \\
\end{split}
\end{equation}
\end{proof}

In order to construct the nonlinear weights with a more accesible notation, we can check the following result.
\begin{proposition}\label{prop:2}
With the notation used in this section,
\begin{equation}
C^{r+l-1}_{l-1,k}=C^{r+l-1}_{r-1,r-1-k}, \quad 1\leq l \leq r-1,\quad 0\leq k \leq l-1.
\end{equation}
\end{proposition}

Then, we define the new nonlinear weights as:
\begin{equation}\label{definicionomegar}
\omega_{k}=\frac{\alpha_{k}}{\sum_{\xi=0}^{r-1}\alpha_{\xi}}, \quad 0\leq k\leq r-1, \end{equation}
with
\begin{equation}\label{definicionalfar}
\alpha_k=C^{2r-1}_{r-1,k}+h^{-t}\sum^{r-1}_{l=k+1}C^{r+l-1}_{l-1,k}J_{l}+h^{-t}\sum_{l=0}^{k-1}C^{2r-l-2}_{r-1,k}J_{-r+1+l},\\
\end{equation}
being
\begin{equation*}
J_{k}=|f_{j+k}-f_{j+k-1}|^{2t},\quad k=-r+1,\dots,r-1.
\end{equation*}
As we have mentioned in Section \ref{sec:newweights}, we take $t=2r-1$. In matricial form, if we denote as $||\cdot||_{\ell^1}$ the $\ell^1$-vector norm and
\begin{equation*}
\begin{split}
\mathbb{A}^r&=(\alpha_0,\alpha_1\dots, \alpha_{r-2},\alpha_{r-1})^T,\\
\mathbb{O}^r&=(C^{2r-1}_{r-1,0},C^{2r-1}_{r-1,1}\dots,C^{2r-1}_{r-1,r-2},C^{2r-1}_{r-1,r-1})^T,\\
\mathbb{J}^r&=(J_{r-1},J_{r-2},\dots,J_{1},J_{-1},\dots,J_{-r+2},J_{-r+1})^T,\\
\mathbb{W}^r&=(\omega_0,\omega_1\dots, \omega_{r-2},\omega_{r-1})^T,
\end{split}
\end{equation*}
we get
\begin{equation}
\begin{split}
\mathbb{W}^r&=\frac{\mathbb{A}^r}{||\mathbb{A}^r||_{\ell^1}}, \quad \text{with}\\
\mathbb{A}^r&=\mathbb{O}^r+h^{-t}(A^{r}|\tilde A^{r})\mathbb{J}^r,
\end{split}
\end{equation}
%
%
being $(A^r|\tilde A^r)$ the $(2r)\times (r-1)$ matrix formed by the $r\times (r-1)$ matrices following:
\begin{equation*}
A^{r}=\left(
  \begin{array}{ccccccc}
    C^{2r-2}_{r-2,0}    & C^{2r-3}_{r-3,0}       & C^{2r-4}_{r-4,0}         &   \hdots   & C^{r+2}_{2,0}    & C^{r+1}_{1,0}  &  C^{r}_{0,0}  \\
    C^{2r-2}_{r-2,1}    & C^{2r-3}_{r-3,1}       & C^{2r-4}_{r-4,1}         &   \hdots   & C^{r+2}_{2,1}    & C^{r+1}_{1,1}  &       0  \\
    C^{2r-2}_{r-2,2}    & C^{2r-3}_{r-3,2}       & C^{2r-4}_{r-4,2}         &   \hdots   & C^{r+2}_{2,2}    & 0             &       0  \\
    \vdots              & \vdots                 & \vdots                   & \udots     & \udots           &   \vdots      &  \vdots    \\
C^{2r-2}_{r-2,r-4}      & C^{2r-3}_{r-3,r-4}     & C^{2r-4}_{r-4,r-4}       & \hdots         & 0                & 0             &  0   \\
C^{2r-2}_{r-2,r-3}      & C^{2r-3}_{r-3,r-3}     & 0                        & \hdots         & 0                & 0             &  0   \\
C^{2r-2}_{r-2,r-2}      & 0                      & 0                        & \hdots         & 0                & 0             &  0   \\
    0                   & 0                      & 0                        & \hdots          & 0                & 0             &  0   \\
  \end{array}
\right),
\end{equation*}
and
\begin{equation*}
\tilde A^{r}=\left(
  \begin{array}{ccccccc}
     0                &  0               &  0                       &   \hdots   & 0    & 0  &  0  \\
     0                &  0               &  0                       &   \hdots   & 0    & 0  &  C^{2r-2}_{r-1,1}  \\
     0                &  0               &  0                       &   \hdots   & 0    & C^{2r-3}_{r-1,2}             &  C^{2r-2}_{r-1,2}  \\
    \vdots            & \vdots           & \vdots                   & \udots     & \udots           &   \vdots      &  \vdots    \\
0                     & 0 & 0    & \hdots          & C^{2r-4}_{r-1,r-4}             & C^{2r-3}_{r-1,r-4}             &  C^{2r-2}_{r-1,r-4}   \\
0                     & 0                           & C^{r+2}_{r-1,r-3}                    &  \hdots          & C^{j+2r-4}_{r-1,r-3}             & C^{2r-3}_{r-1,r-3}             &  C^{2r-2}_{r-1,r-3}   \\
0   & C^{r+1}_{r-1,r-2}         & C^{r+2}_{r-1,r-2}                    &  \hdots          & C^{2r-4}_{r-1,r-2}             & C^{2r-3}_{r-1,r-2}             &  C^{2r-2}_{r-1,r-2}   \\
   C^{r}_{r-1,r-1}    & C^{r+1}_{r-1,r-1}           & C^{r+2}_{r-1,r-1}&  \hdots          & C^{2r-4}_{r-1,r-1}               & C^{2r-3}_{r-1,r-1}         &  C^{2r-2}_{r-1,r-1}     \\
  \end{array}
\right).
\end{equation*}
By Prop. \ref{prop:2} we can rewrite it as:
\begin{equation*}
\tilde A^{r}=\left(
  \begin{array}{ccccccc}
     0                &  0               &  0                       &   \hdots   & 0    & 0  &  0  \\
     0                &  0               &  0                       &   \hdots   & 0    & 0  &  C^{2r-2}_{r-2,r-2}    \\
     0                &  0               &  0                       &   \hdots   & 0    & C^{2r-3}_{r-3,r-3}             &  C^{2r-2}_{r-2,r-3}   \\
    \vdots            & \vdots           & \vdots                   & \udots     & \udots           &   \vdots      &  \vdots    \\
0                     & 0 & 0    & \hdots          & C^{2r-4}_{r-4,3}              & C^{2r-3}_{r-3,3}             &  C^{2r-2}_{r-2,3}     \\
0                     & 0    & C^{r+2}_{2,2}   &  \hdots  & C^{2r-4}_{r-4,2} & C^{2r-3}_{r-3,2}     &  C^{2r-2}_{r-2,2}    \\
0                     & C^{r+1}_{1,1}  & C^{r+2}_{2,1}    &  \hdots & C^{2r-4}_{r-4,1}  & C^{2r-3}_{r-3,1}            &  C^{2r-2}_{r-2,1}     \\
   C^{r}_{0,0}    & C^{r+1}_{1,0} & C^{r+2}_{2,0}    &  \hdots          & C^{2r-4}_{r-4,0}  & C^{2r-3}_{r-3,0}    &  C^{2r-2}_{r-2,0}     \\
  \end{array}
\right).
\end{equation*}

We introduce the following propositions as auxiliary results to prove that the obtained order of accuracy is maximum near the discontinuities.

%
\begin{proposition}\label{prop:pesosr0}
If $f$ is smooth in the interval $[x_{j-r},x_{j+r-1}]$ then
\begin{equation*}
\omega_{k}=C^{2r-1}_{r-1,k} + O(h^t), \quad 0\leq k \leq r-1,
\end{equation*}
with $t=2r-1$.
\begin{proof}
By hypothesis, if $f$ is smooth in the interval $[x_{j-r},x_{j+r-1}]$ then $J_\xi=O(h^{2t})$, $\xi=-r+1,\dots, r-1$.
Since Eq. \eqref{definicionalfar}, we obtain:
\begin{equation}
\sum_{k=0}^{r-1} \alpha_k = 1+h^{-t}\sum^{r-1}_{l=k+1}J_{l}+h^{-t}\sum_{l=0}^{k-1}J_{-r+1+l}=1+h^{-t}\sum_{l=-r+1,l\neq k}^{r-1}J_{l}=1+O(h^t)
\end{equation}
then let be $0\leq k \leq r-1$
\[
\alpha_k=C^{2r-1}_{r-1,k}+h^{-t}\sum^{r-1}_{l=k+1}C^{r+l-1}_{l-1,k}J_{l-1}+h^{-t}\sum_{l=0}^{k-1}C^{2r-l-2}_{r-1,k}J_{-r+1+l}=C^{2r-1}_{r-1,k}+O(h^t).
\]
Thus, by Eq. \eqref{definicionomegar}
\[
\omega_k=\frac{\alpha_k}{\sum_{\xi=0}^{r-1} \alpha_\xi}=\frac{C^{2r-1}_{r-1,k}+O(h^t)}{1+O(h^t)}=C^{2r-1}_{r-1,k}+O(h^t).
\]

\end{proof}

\end{proposition}
\begin{proposition}\label{prop:pesosr1}
Let be $1\leq l_0\leq r-1$, if $f$ has a discontinuity at the interval $[x_{j+l_0-1},x_{j+l_0}]$ and  it is smooth in rest of the intervals then
\begin{equation*}
\omega_{k}=\left\{
                \begin{array}{ll}
                  C^{r+l_0-1}_{l_0-1,k} + O(h^t), & 0\leq k \leq l_0-1; \\
                  O(h^t), &l_0\leq k \leq r.
                \end{array}
              \right.
\end{equation*}
with $t=2r-1$.
\end{proposition}
\begin{proof}
Let be a fixed value $1\leq l_0\leq r-1$, then we denote as $[J_{l_0}]$ the value of the jump because, by hypothesis, there exists a discontinuity at the interval $[x_{j+l_0-1},x_{j+l_0}]$ and let be $0 \leq k \leq l_0-1$,  then
\begin{equation*}
\begin{split}
\alpha_k=&C^{2r-1}_{r-1,k}+h^{-t}\sum^{r-1}_{l=k+1}C^{r+l-1}_{l-1,k}J_{l}+h^{-t}\sum_{l=0}^{k-1}C^{2r-l-2}_{r-1,k}J_{-r+1+l}\\
=&C^{2r-1}_{r-1,k}+h^{-t}C^{r+l_0-1}_{l_0-1,k}[J_{l_0}]+O(h^{t}),
\end{split}
\end{equation*}
and
\begin{equation*}
\begin{split}
\sum_{k=0}^{r-1}\alpha_k=&1+h^{-t}\sum^{l_0-1}_{k=0}C^{r+l_0-1}_{l_0-1,k}[J_{l_0}]+O(h^{t})=1+h^{-t}[J_{l_0}]+O(h^{t}),\\
\end{split}
\end{equation*}
then
\begin{equation*}
\begin{split}
\omega_k-C^{r+l_0-1}_{l_0-1,k}&=\frac{\alpha_k}{\sum_{\xi=0}^{r-1} \alpha_\xi}-C^{r+l_0-1}_{l_0-1,k}=\frac{C^{2r-1}_{r-1,k}+h^{-t}C^{r+l_0-1}_{l_0-1,k}[J_{l_0}]+O(h^{t})}{1+h^{-t}[J_{l_0}]+O(h^{t})}-C^{r+l_0-1}_{l_0-1,k}\\
&=\frac{C^{r+l_0-1}_{l_0-1,k}[J_{l_0}]+h^t C^{2r-1}_{r-1,k}+O(h^{2t})}{[J_{l_0}]+h^{t}+O(h^{2t})}-C^{r+l_0-1}_{l_0-1,k}=\frac{h^t C^{2r-1}_{r-1,k}+O(h^{2t})}{[J_{l_0}]+h^{t}+O(h^{2t})}\\
&=O(h^t).
\end{split}
\end{equation*}

Finally, we suppose that $l_0\leq k \leq r-1$ so, if $l_0<k+1\leq l_1 \leq r-1$ then $J_{l_1}=O(h^{2t})$ and if $0\leq l_2 \leq k-1 < r-1$ then $-r+1+l_2<0$, hence $J_{l_2}=O(h^{2t})$.
\end{proof}

The following proposition is a result similar to Prop. \ref{prop:pesosr1}. Both propositions (Props. \ref{prop:pesosr1} and \ref{prop:pesosr2}) allow us to prove the next theorem that it is the central result of the paper.

\begin{proposition}\label{prop:pesosr2}
Let be $1\leq l_0\leq r-1$, if $f$ has a discontinuity at the interval $[x_{j-l_0-1},x_{j-l_0}]$ and  it is smooth in rest of the intervals then
\begin{equation*}
\omega_{k}=\left\{
                \begin{array}{ll}
                   O(h^t), & 0\leq k \leq r-l_0-1; \\
                   C^{r-l_0-1}_{r-1,k} + O(h^t), & r-l_0\leq k \leq r-1,
                \end{array}
              \right.
\end{equation*}
with $t=2r-1$.
\end{proposition}

\begin{theorem}\label{teo1}
Let be $1\leq l_0 \leq r-1$, if $f$ is smooth in $[x_{j-r},x_{j+r-1}]\setminus \Omega$ and $f$ has a discontinuity at $\Omega$ then
\begin{equation}
\hat{f}_{j-\frac12}-{f}_{j-\frac12}=\left\{
                                                  \begin{array}{ll}
                                                    O(h^{2r}), & \hbox{if $\,\,\Omega=\emptyset$;} \\
O(h^{r+l_0}), & \hbox{if  $\,\,\Omega=[x_{j-l_0-1},x_{j-l_0}]$ or $\,\,\Omega=[x_{j+l_0-1},x_{j+l_0}]$; }\\
                                                  \end{array}
                                                \right.
\end{equation}
\end{theorem}

\begin{proof}
Firstly, if $f$ is smooth in $[x_{j-r},x_{j+r-1}]$, then
 \begin{equation*}
    p^{2r-1}_{r-1}(x_{j-\frac12}) = f_{j-\frac12}+O(h^{2r}),\quad p^{r}_{k}(x_{j-\frac1 2}) = f_{j-\frac12}+O(h^{r+1}), \quad 0\leq k \le r-1,
  \end{equation*}
and using Proposition \ref{prop:pesosr0} we can obtain  the order of accuracy:
\begin{equation*}
\begin{split}
\sum_{k=0}^{r-1} \omega_{k}p^r_{k}(x_{j-\frac12})-p^{2r-1}_{r-1}(x_{j-\frac12})
&=\sum_{k=0}^{r-1}\omega_{k}p^r_{k}(x_{j-\frac12})
-\sum_{k=0}^{r-1}C^{2r-1}_{r-1,k}p^r_{k}(x_{j-\frac12})\\
&=\sum_{k=0}^{r-1}(\omega_k-C^{2r-1}_{r-1,k})p^r_{k}(x_{j-\frac12})\\
&=\sum_{k=0}^{r-1} (\omega_{k}-C^{2r-1}_{r-1,k})(p^r_{k}(x_{j-\frac12})- f_{i-\frac12})\\
&=O(h^t)\cdot O(h^{r+1})=O(h^{t+r+1}).
\end{split}
\end{equation*}
Then,
\begin{eqnarray*}
\hat{f}_{j-\frac12}-{f}_{j-\frac12}&=&
\sum_{k=0}^{r-1} \omega_{k}p^r_{k}(x_{j-\frac12})-p^{2r-1}_{r-1}(x_{j-\frac12})
+(p^{2r-1}_{r-1}(x_{j-\frac12})-{f}_{j-\frac12})\\
&=&O(h^{t+r+1})+O(h^{2r})=O(h^{2r}).
\end{eqnarray*}

Secondly, we suppose that it exists a value $1\leq l_0\leq r-1$ such that $f$ has a discontinuity at $[x_{j+l_0-1},x_{j+l_0}]$ (analogously for $[x_{j-l_0-1},x_{j-l_0}]$). Since:
 \begin{equation*}
 p^{r+l_0-1}_{l_0-1}(x_{j-\frac1 2}) = f_{j-\frac12}+O(h^{r+l_0}),
  \end{equation*}
and, by Prop. \ref{prop:pesosr1}, we get
\begin{equation*}
\begin{split}
\sum_{k=0}^{r-1} \omega_{k}p^r_{k}(x_{j-\frac12})-f_{j-\frac1 2}=&\sum_{k=0}^{r-1} \omega_{k}p^r_{k}(x_{j-\frac12})-p^{r+l_0-1}_{l_0-1}(x_{j-\frac12})+(f_{j-\frac1 2} - p^{r+l_0-1}_{l_0-1}(x_{j-\frac12}))\\
=&\sum_{k=0}^{r-1}\omega_{k}p^r_{k}(x_{j-\frac12})
-\sum_{k=0}^{l_0-1}C^{r+l_0-1}_{l_0-1,k}p^r_{k}(x_{j-\frac12})\\
&+(f_{j-\frac1 2} - p^{r+l_0-1}_{l_0-1}(x_{j-\frac12}))\\
=&\sum_{k=0}^{l_0-1}(\omega_{k}-C^{r+l_0-1}_{l_0-1,k})p^r_{k}(x_{j-\frac12})
+\sum_{k=l_0}^{r-1}\omega_{k}p^r_{k}(x_{j-\frac12})\\
&+(f_{j-\frac1 2} - p^{r+l_0-1}_{l_0-1}(x_{j-\frac12}))\\
=&\sum_{k=0}^{l_0-1}(\omega_{k}-C^{r+l_0-1}_{l_0-1,k})(p^r_{k}(x_{j-\frac12})-f_{j-\frac1 2})
+\sum_{k=l_0}^{r-1}\omega_{k}(p^r_{k}(x_{j-\frac12})-f_{j-\frac1 2})\\
&+(f_{j-\frac1 2} - p^{r+l_0-1}_{l_0-1}(x_{j-\frac12}))\\
=&O(h^t)O(h^{r+1})+O(h^t)O(1)+O(h^{r+l_0})=O(h^{r+l_0}).\\
\end{split}
\end{equation*}
\end{proof}

\subsection{Design of new nonlinear weights: The case $r=4$}
\label{sec:newweightsr4}
In order to clarify the construction for $2r$, we introduce the case $r=4$ ($r=3$ has been showed in Section \ref{sec:newweights}). We calculate the matrix $(A^4|\tilde A^4)$ using Theorem \ref{teoremaweightsr}, thus
$$(A^4|\tilde A^4)=
\left(
  \begin{array}{cccccc}
     3/24  & 3/10 & 1 & 0 & 0            & 0 \\
     14/24 & 7/10 & 0 & 0 & 0            & 7/24 \\
     7/24  & 0            & 0 & 0 & 7/10 & 14/24 \\
     0             & 0            & 0 & 1 & 3/10 & 3/24 \\
  \end{array}
\right),
$$
and using Prop. \ref{Propox1}, the vector
$$\mathbb{O}^4=(C^{7}_{3,0},C^{7}_{3,1},C^{7}_{3,2},C^{7}_{3,4})^T=\frac{1}{16}(1,7,7,1).$$

The following results are a direct consequence of Prop. \ref{prop:pesosr2} and Theo. \ref{teo1}

\begin{corollary}\label{corollary1r4}
If $f$ is smooth in $[x_{j-4},x_{j+3}]\setminus \Omega$ and $f$ has a discontinuity at $\Omega$ then the following results are satisfied:

 \begin{equation*}
\begin{array}{c|cccc}
 \Omega & \omega_{0} & \omega_{1} & \omega_{2} & \omega_{3} \\\hline
\emptyset            & C^{7}_{3,0}+O(h^t) & C^{7}_{3,1}+O(h^t) & C^{7}_{3,2}+O(h^t) & C^{7}_{3,3}+O(h^t) \\
  \text{$[x_{j+2},x_{j+3}]$}    & C^{6}_{2,0}+O(h^t) & C^{6}_{2,1}+O(h^t) & C^{6}_{2,2}+O(h^t) & O(h^t) \\
\text{$[x_{j+1},x_{j+2}]$}    & C^{5}_{1,0}+O(h^t) & C^{5}_{1,1}+O(h^t) & O(h^t)  & O(h^t) \\
\text{$[x_{j},x_{j+1}]$}      & C^{4}_{0,0}+O(h^t) & O(h^t)  & O(h^t)  & O(h^t)  \\
\text{$[x_{j-2},x_{j-1}]$}    &  O(h^t) &  O(h^t) &  O(h^t)& C^{4}_{3,3}+O(h^t) \\
\text{$[x_{j-3},x_{j-2}]$}    &  O(h^t) &  O(h^t) & C^{5}_{3,2}+O(h^t) & C^{5}_{3,3}+O(h^t) \\
\text{$[x_{j-4},x_{j-3}]$}    & O(h^t)  & C^{6}_{3,1}+O(h^t) & C^{6}_{3,2}+O(h^t) & C^{6}_{3,3}+O(h^t) \\
\end{array}
   \end{equation*} \end{corollary}

Finally, we show the following corollary to summarize our example. The proof is a direct consequence of Theorem \ref{teo1}.
\begin{corollary}
If $f$ is smooth in $[x_{j-4},x_{j+3}]\setminus \Omega$ and it has a discontinuity at $\Omega$ then
\begin{equation}
\hat{f}_{j-\frac12}-{f}_{j-\frac12}=\left\{
                                                  \begin{array}{ll}
                                                    O(h^8), & \hbox{if $\Omega=\emptyset $;} \\
                                                    O(h^7), & \hbox{if $\Omega=[x_{j-4},x_{j-3}] \text{ or } \Omega=[x_{j+1},x_{j+2}]$;} \\
                                                    O(h^6), & \hbox{if $\Omega=[x_{j-3},x_{j-2}] \text{ or } \Omega=[x_{j+2},x_{j+3}]$. } \\
O(h^5), & \hbox{if $\Omega=[x_{j-2},x_{j-1}] \text{ or } \Omega=[x_{j},x_{j+1}]$. }
                                                  \end{array}
                                                \right.
\end{equation}
\end{corollary}

\section{Numerical experiments}\label{sec:numerexperiments}

In this section we present a comparison between the classical linear method, Eq. \eqref{eq:linear}, WENO interpolation (see \cite{liu:1,belda:2}) presented in Section \ref{sec:weno} and the new method designed with $r=3$ and $r=4$. In particular, we calculate an approximation to the order at the points close to the discontinuities using the following formula:
\begin{equation}\label{eqaprreorders}
o(x^i_{j}):=o^i_j=\log_2\left(\frac{e^i_{j}}{e^{i+1}_{j}}\right),
\end{equation}
where $e^i_{j}=|f(x^{i}_j)-\hat f^i_j|$ is the absolute error at point $x_j^i$. We use the functions plotted in Figure \ref{fig:funciones}. Both functions present an isolated jump discontinuity. We perform these experiments because the piecewise exponential function $f_1$ is used by Ar\`andiga in \cite{arandiga:jcam2019} and the function $f_2$ is one of the examples proposed in \cite{amat:weno1,amat:weno2}.
\begin{figure}
\begin{center}
\begin{tabular}{ccc}
\includegraphics[width=6cm,height=4cm]{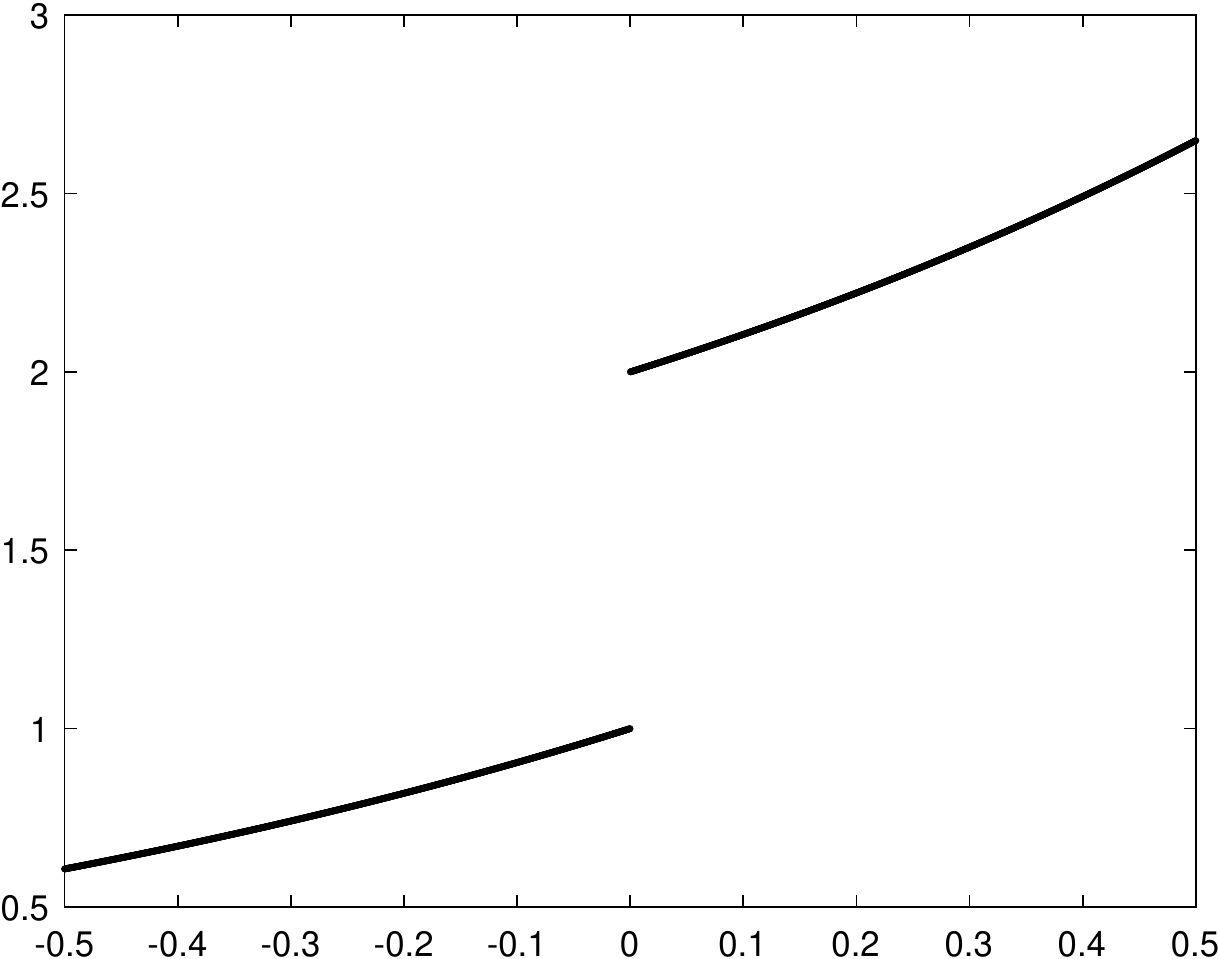}& \includegraphics[width=6cm,height=4cm]{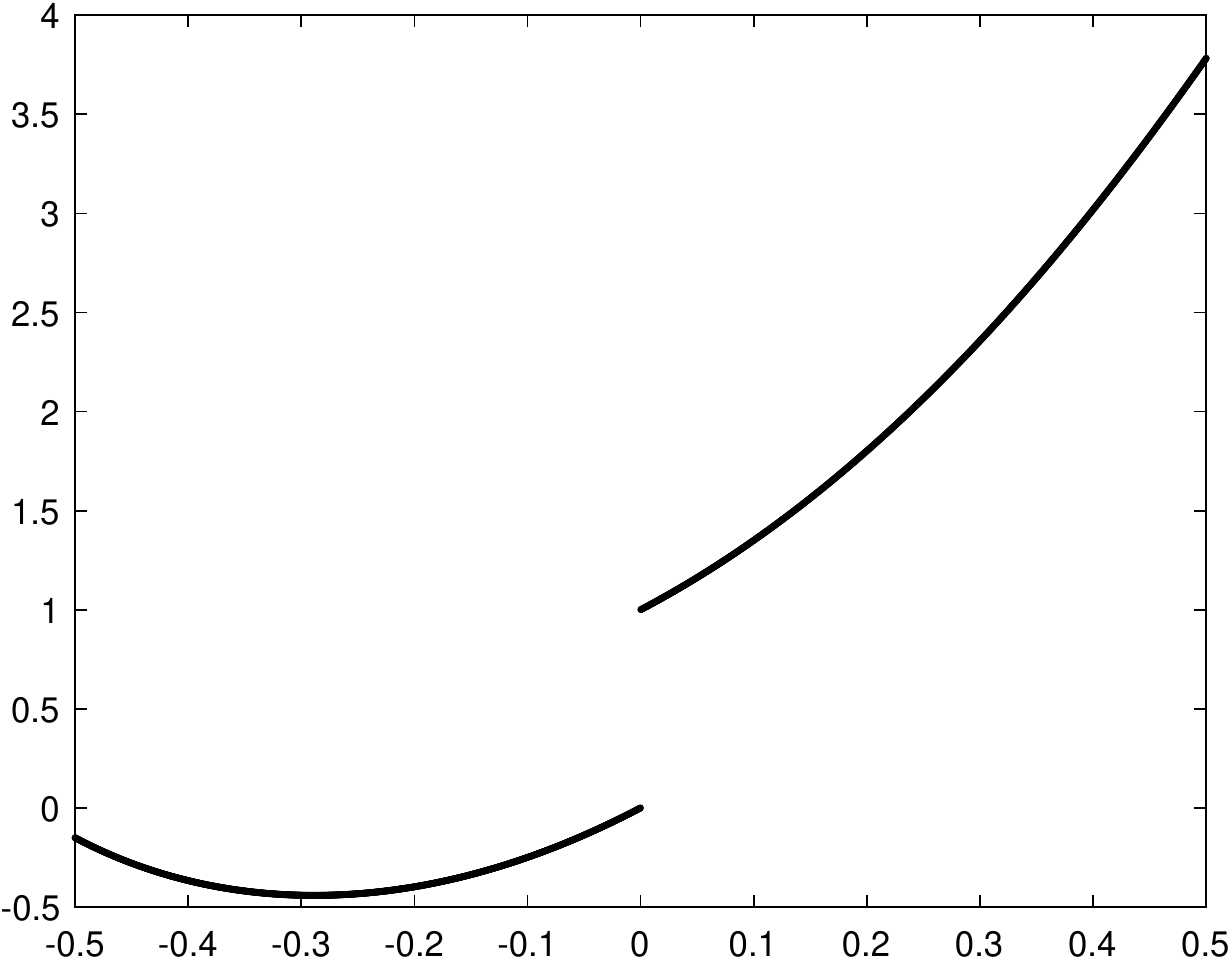}\\
a) & b) \\
\end{tabular}
\end{center}
\caption{Functions: a) $f_1(x),\,x\in[-0.5,0.5]$,  b) $f_2(x),\,x\in[-0.5,0.5]$}
\label{fig:funciones}
\end{figure}

\emph{Example 1.} Let's start with the function $f_1$,
\begin{equation*}
f_1(x)=\left\{
       \begin{array}{ll}
         e^x &   x\leq 0; \\
         1 + e^x, & 0 < x,
       \end{array}
     \right.
\end{equation*}
we consider the values of the function discretized in the interval $[-0.5,0.5]$ for $r=3$ and $[-2,2]$ for $r=4$, being $\{f_1(x^{i}_j)\}^{N_i}_{j=0}$, with $N_i=2^i$, $i=4,\dots,8$. We show the errors and the values $o(x^i_{j})$, Eq. \eqref{eqaprreorders}, at the points around of the interval which contains the discontinuity, we denote it as $[x_{j_0-1}^{i},x_{j_0}^{i}]$. We only show the error and the approximation order at the points $\{x^{i+1}_{(2j_0-1)+2\xi}\}_{\xi=0}^4$ in Tables \ref{tablaf1r3} and \ref{tablaf1r4}. The results in the symmetric points are similar. When $r=3$, if we use the linear method, the order is degraded at the points $x^{i+1}_{(2j_0-1)+2}$, $x^{i+1}_{(2j_0-1)+4}$ because the stencil crosses the discontinuity.  We can check that at the points $x^{i+1}_{(2j_0-1)+2}$, the order obtained using WENO and rational techniques is similar as well $r=3$ as $r=4$. However, at adjacent points, $x^{i+1}_{(2j_0-1)+4}$ when $r=3$ and also $x^{i+1}_{(2j_0-1)+6},x^{i+1}_{(2j_0-1)+8}$ when $r=4$, the order obtained using rational interpolant is higher than using WENO method. WENO is designed to avoid oscillations in non-smooth zones but the accuracy of the interpolant decreases to order $r+1$ around the central interval.
The order obtained in rational method is increasing using all the free discontinuities information as we prove in Sections \ref{sec:newweights2r} and \ref{sec:newweightsr4}.

\emph{Example 2.} We take the function
 \begin{equation*}
f_2(x)=\left\{
       \begin{array}{ll}
         -x^9 + x^8-4x^7 + x^4 + 5x^2 + 3x, &  x< 0; \\
         -x^9 + x^8-4x^7 + x^4 + 5x^2 + 3x+1, & 0\leq x,
       \end{array}
     \right.
\end{equation*}
which is discretized as $f_1$. In this case, the differences between rational and WENO methods can be noted at the point $x^{i+1}_{(2j_0-1)+4}$ when $r=3$ (Table \ref{tablaf2r3}) but there are no differences when $r=4$ (Table \ref{tablaf2r4}), because the function is piecewise polynomial of degree 9.


\section{Conclusions and future work}\label{sec:conclusions}

In this paper, we present a generalization of the adaptive rational method designed by \cite{ramponi:1,carrato:1,castagno:1,castagno:2} and modified in \cite{arandiga:jcam2019}. We construct this new method in the same way that the WENO method (see \cite{liu:1,belda:2}) introducing a new concept of optimal weights depending
on the location of the jump discontinuity. This new method only detects discontinuities in the original function nor in first derivative (corner discontinuities). It is a crucial property in image processing context. For this reason, it will be an interesting future work to analyze adaptations of the new weights to discontinuities in the first derivative. The principal property of the new adaptive rational method is that, if it is possible, it uses all the information which does not contain discontinuities to interpolate, increasing the order of accuracy in the points close to discontinuity. On the other hand, we can reproduce the same construction with others settings as cell-average or hat-average. The numerical examples presented confirm the theoretical results proved.

\bibliographystyle{unsrt}
\bibliography{references}

\begin{thebibliography}{10}

\bibitem{amat:weno1}
S.~Amat, J.~Ruiz, and C.-W. Shu.
\newblock On new strategies to control the accuracy of weno algorithms close to
  discontinuities.
\newblock {\em SIAM J. Numer. Anal.}, 57:1205--1237, 2019.

\bibitem{arandiga:jcam2019}
F.~Ar\`andiga.
\newblock Adaptive rational interpolation for point values.
\newblock {\em J. Comput. App. Math.}, 349:221--224, 2019.

\bibitem{belda:1}
F.~Ar\`andiga and A.~Belda.
\newblock Weighted {ENO} interpolation and applications.
\newblock {\em Commun. Nonlinear Sci.}, 9(2):187--195, 2003.

\bibitem{belda:2}
F.~Ar\`andiga, A.~Belda, and P.~Mulet.
\newblock Point-value {WENO} multiresolution applications to stable image.
\newblock {\em J. Sci. Comput.}, 43(2):158--182, 2010.

\bibitem{arandiga:1}
F.~Ar\`andiga and R.~Donat.
\newblock Nonlinear multiscale decompositions: The approach of {A}mi {H}arten.
\newblock {\em Numer. Algor.}, 23:175--216, 2000.

\bibitem{harten:2}
A.~Harten.
\newblock Discrete multiresolution analysis and generalized wavelets.
\newblock {\em App. Num. Math.}, 12:153--193, 1993.

\bibitem{liu:1}
X-D. Liu, S.~Osher, and T.~Chan.
\newblock Weighted essentially non-oscillatory schemes.
\newblock {\em J. Comput. Physics}, 115:200--212, 1994.

\bibitem{harten:3}
A.~Harten.
\newblock Multiresolution representation of data: a general framework.
\newblock {\em SIAM J. Numer. Anal.}, 33(3):1205--1256, 1996.

\bibitem{arandiga:2}
F.~Ar\`andiga, R.~Donat, and A.~Harten.
\newblock Multiresolution based on weighted averages of the hat function i:
  Linear reconstruction techniques.
\newblock {\em SIAM J. Numer. Anal.}, 36(1):160--203, 1999.

\bibitem{harten:1}
A.~Harten, B.~Engquist, S.~Osher, and S.~Chakravarthy.
\newblock Multiresolution representation of data: a general framework.
\newblock {\em J. Comput. Physics}, 71:231--303, 1987.

\bibitem{amatruizshuyanez}
S.~Amat, J.~Ruiz, C.-W. Shu, and D.~F. Y\'anez.
\newblock A new weno-2r algorithm with progressive order of accuracy close to
  discontinuities.
\newblock {\em SIAM J. Numer. Anal.}, 2020.

\bibitem{shu:1}
C.-W. Shu.
\newblock High order weighted essentially nonoscillatory schemes for convection
  dominated problems.
\newblock {\em SIAM Rev.}, 51:82--126, 2009.

\bibitem{carrato:1}
R.~Castagno, S.~Marsi, and G.~Ramponi.
\newblock Interpolation of the dc component of coded images using a rational
  filter.
\newblock {\em Proc. Fourth IEEE Intern. Conf. on Image Processing - 97}, pages
  26--29, 1997.

\bibitem{castagno:1}
R.~Castagno and G.~Ramponi.
\newblock A rational filter for the removal of blocking artifacts in image
  sequences coded at low bitrate.
\newblock {\em Proc. Eighth European Signal Processing Conf., EUSIPCO-96}.

\bibitem{castagno:2}
R.~Castagno, S.~Marsi, and G.~Ramponi.
\newblock A simple algorithm for the reduction of blocking artifacts in images
  and its implementation.
\newblock {\em IEEE Trans. on Consumer Electronics}, 44(3):1062--1070, 1998.

\bibitem{ramponi:1}
G.~Ramponi.
\newblock Image processing using rational functions.
\newblock {\em Proc. Second IEEE Intern. Conf. on Image Processing - 95}, pages
  22--25, 1995.

\bibitem{arandiga.yanez:1}
F.~Ar\`andiga and D.~F. Y\'anez.
\newblock Adaptive rational interpolation for cell average.
\newblock {\em App. Math. Letters}, 2020.

\bibitem{amat:weno2}
S.~Amat, J.~Ruiz, and C.-W. Shu.
\newblock On a new weno algorithm of order $2r$ with improved accuracy close to
  discontinuities.
\newblock {\em App. Math. Letters}, 105, 2020.

\end{thebibliography}

\begin{table}[H]
{\footnotesize
\begin{equation*}
\begin{array}{|c|cc|cc|cc|cc|}
\multicolumn{9}{c}{\text{Linear interpolation, 6-points} }\\
\hline
    &    e^{i+1}_{2j_0+1}        & o^{i+1}_{2j_0+1}& e^{i+1}_{2j_0+3}        & o^{i+1}_{2j_0+3}& e^{i+1}_{2j_0+5}        & o^{i+1}_{2j_0+5} & e^{i+1}_{2j_0+7}        & o^{i+1}_{2j_0+7}\\
\hline
     {0.5}                 & 8.59e-02&        & 1.17e-02&   &  3.76e-10  &    &  3.85e-10 &                    \\
    {0.5}      & 8.59e-02& 5.2e-09& 1.17e-02& -4.1e-08&     5.08e-12 &6.15&  5.23e-12&  6.20                \\
    {0.5}      & 8.59e-02& 7.8e-11& 1.17e-02& -5.9e-10&     7.54e-14 &6.07&  7.64e-14&  6.09                             \\
    {0.5}      & 8.59e-02& 1.2e-12& 1.17e-02& -9.0e-12&     1.14e-15 &6.04&  1.27e-15&  5.90                                     \\
\hline
\multicolumn{9}{c}{\text{Rational interpolation,} \quad r=3}\\
\hline
  &    e^{i+1}_{2j_0+1}        & o^{i+1}_{2j_0+1}& e^{i+1}_{2j_0+3}        & o^{i+1}_{2j_0+3}& e^{i+1}_{2j_0+5}        & o^{i+1}_{2j_0+5} & e^{i+1}_{2j_0+7}        & o^{i+1}_{2j_0+7}\\
\hline
 {0.5}               &     7.32e-07  &       &  7.43e-09  &       &  3.76e-10  &      &      4.14e-10   &                      \\
    {0.5}       &    4.19e-08 &  4.12 &     1.27e-10 &      5.87 &   5.08e-12 &  6.21&       5.24e-12 &    6.30                \\
    {0.5}       &    2.48e-09 &  4.07 &     2.08e-12 &   5.92 &   7.54e-14 &  6.07&       7.64e-14 &    6.09                             \\
    {0.5}      &   1.50e-10 &  4.04 &   3.29e-14 &  \textbf{5.98} &   1.14e-15 &  6.04&       1.27e-15 &    5.90                                     \\
\hline
\multicolumn{9}{c}{\text{WENO}, r=3. }\\
\hline
  &    e^{i+1}_{2j_0+1}        & o^{i+1}_{2j_0+1}& e^{i+1}_{2j_0+3}        & o^{i+1}_{2j_0+3}& e^{i+1}_{2j_0+5}        & o^{i+1}_{2j_0+5} & e^{i+1}_{2j_0+7}        & o^{i+1}_{2j_0+7}\\
\hline
  {0.49} &  9.34e-06  &        &   8.12e-07 &      &   3.67e-10&      &   3.91e-10 &      \\
  {0.49} &  5.29e-07  &    4.13&   4.41e-08 &  4.20&   5.09e-12  &  6.17&   5.25e-12 &  6.22\\
  {0.49} &  3.15e-08  &    4.06&   2.57e-09 &  4.09&   7.54e-14&  6.07&   7.65e-14 &  6.10\\
  {0.49} &  1.92e-09  &    4.03&   1.55e-10 &  4.04&   1.14e-15 &  6.04&   1.27e-15 &  5.90\\
 \hline
\end{array}
\end{equation*}
\caption{Errors, $e^{i+1}_{(2j_0-1)+\xi}$, and orders of approximation, $o^{i+1}_{(2j_0-1)+\xi}$, with $\xi=2,4,5,8$ for the different methods, with $f_1(x), x \in [-0.5,0.5]$}
\label{tablaf1r3}
}
\end{table}



\begin{table}[H]
{\footnotesize
\begin{equation*}
\begin{array}{|c|cc|cc|cc|cc|}
\multicolumn{9}{c}{\text{Linear interpolation, 6-points} }\\
\hline
     &    e^{i+1}_{2j_0+1}        & o^{i+1}_{2j_0+1}& e^{i+1}_{2j_0+3}        & o^{i+1}_{2j_0+3}& e^{i+1}_{2j_0+5}        & o^{i+1}_{2j_0+5} & e^{i+1}_{2j_0+7}        & o^{i+1}_{2j_0+7}\\
\hline
        {0.5}               &   8.59e-02   &         & 1.17e-02   &          &  8.43e-07   &        & 1.19e-06  &      \\
    {0.5}       &     8.59e-02 &  -2.8e-06 & 1.17e-02 &   6.2e-05&    6.72e-09 &   6.97 &   9.28e-09&    7.00\\
    {0.5}       &     8.59e-02 &  -2.3e-08 & 1.17e-02 &   5.0e-07&    5.39e-11 &   6.96 &   7.47e-11&    6.95\\
    {0.5}       &     8.59e-02 &  -1.8e-10 & 1.17e-02 &   4.0e-09&    4.28e-13 &   6.97 &   5.96e-13&    6.96\\
\hline
\multicolumn{9}{c}{\text{Rational interpolation,} \quad r=3}\\
\hline
   &    e^{i+1}_{2j_0+1}        & o^{i+1}_{2j_0+1}& e^{i+1}_{2j_0+3}        & o^{i+1}_{2j_0+3}& e^{i+1}_{2j_0+5}        & o^{i+1}_{2j_0+5} & e^{i+1}_{2j_0+7}        & o^{i+1}_{2j_0+7}\\
\hline
     {0.45}            &  1.30e-05 &       &  1.73e-06 &       &     3.95e-06 &        & 3.51e-06 &           \\
   {0.49} &  8.84e-07 &  3.88 &  1.31e-08 &  7.04 &  3.08e-08 &  7.00  & 5.87e-08 &  5.90\\
   {0.49} &    5.58e-08 &  3.98 &  1.10e-10 &  6.89 &  3.11e-11&9.95  & 3.94e-11 &  10.54 \\
  {0.50} &  3.49e-09 &  3.99 &  1.08e-12 & \textbf{6.66} &  4.05e-13 &  6.26  & 5.66e-13 &  6.12 \\
\hline
\multicolumn{9}{c}{\text{WENO}, r=3. }\\
\hline
     &    e^{i+1}_{2j_0+1}        & o^{i+1}_{2j_0+1}& e^{i+1}_{2j_0+3}        & o^{i+1}_{2j_0+3}& e^{i+1}_{2j_0+5}        & o^{i+1}_{2j_0+5} & e^{i+1}_{2j_0+7}        & o^{i+1}_{2j_0+7}\\
\hline
     {0.5}  &      2.83e-04 &      &   3.33e-05  &     &   8.43e-07 &      &   1.19e-06 &      \\       {0.5}  &      1.45e-05 &  4.28&   1.45e-06  & 4.51&   6.71e-09 &  6.97&   9.28e-09 &  7.00\\      {0.5}  &      8.22e-07 &  4.14&   7.41e-08  & 4.29&   5.39e-11 &  6.96&   7.46e-11 &  6.95\\
     {0.5}  &      4.91e-08 &  4.06&   4.16e-09  & 4.15&   4.28e-13 &  6.97&   5.96e-13 &  6.96\\   \hline
\end{array}
\end{equation*}
\caption{Errors, $e^{i+1}_{(2j_0-1)+\xi}$, and orders of approximation, $o^{i+1}_{(2j_0-1)+\xi}$, with $\xi=2,4,5,8$ for the different methods, with $f_2(x), x \in [-0.5,0.5]$}
\label{tablaf2r3}
}
\end{table}

\begin{table}[H]
{\footnotesize
\begin{equation*}
\begin{array}{|c|cc|cc|cc|cc|}
\multicolumn{9}{c}{\text{Linear interpolation, 8-points} }\\
\hline
     &    e^{i+1}_{2j_0+1}        & o^{i+1}_{2j_0+1}& e^{i+1}_{2j_0+3}        & o^{i+1}_{2j_0+3}& e^{i+1}_{2j_0+5}        & o^{i+1}_{2j_0+5} & e^{i+1}_{2j_0+7}        & o^{i+1}_{2j_0+7}\\
\hline
  {0.5}                                           &   9.8e-02 &         &  2.1e-02 &         &  2.4e-03 &         &  5.0e-08 &            \\
    {0.5} &    9.8e-02 & -3.5e-07 &  2.1e-02 &  2.0e-06 &  2.4e-03 & -2.3e-05 &  1.1e-10 &  8.8 \\
     {0.5} &     9.8e-02 & -1.1e-09 &  2.1e-02 &  5.8e-09 &  2.4e-03 & -5.8e-08 &  3.3e-13 &  8.4\\
  {0.5} &    9.8e-02 & -3.9e-12 &  2.1e-02 &  1.9e-11 &  2.4e-03 & -1.8e-10 &  1.3e-15 &  7.9 \\
\hline
\multicolumn{9}{c}{\text{Rational interpolation,} \quad r=4}\\
\hline
    &    e^{i+1}_{2j_0+1}        & o^{i+1}_{2j_0+1}& e^{i+1}_{2j_0+3}        & o^{i+1}_{2j_0+3}& e^{i+1}_{2j_0+5}        & o^{i+1}_{2j_0+5} & e^{i+1}_{2j_0+7}        & o^{i+1}_{2j_0+7}\\
\hline
    {0.5} &     6.0e-05 &     &   1.0e-03  &    &   2.4e-03 &       &  3.3e-07 &        \\     {0.5} &     1.1e-06 &  5.6&   4.1e-08  & 1.4&   1.6e-09 &  20.4 &  5.7e-09 &  5.8   \\
    {0.5} &     3.0e-08 &  5.2&   5.3e-10  & 6.2&   7.8e-12 &  7.7  &  3.1e-13 &  14.1  \\    {0.5} &     8.8e-10 &  5.1&   7.4e-12  & \textbf{6.1}&   3.4e-14 &  \textbf{7.8}  &  1.3e-15 &  7.9   \\
\hline
\multicolumn{9}{c}{\text{WENO}, r=4. }\\
\hline
     &    e^{i+1}_{2j_0+1}        & o^{i+1}_{2j_0+1}& e^{i+1}_{2j_0+3}        & o^{i+1}_{2j_0+3}& e^{i+1}_{2j_0+5}        & o^{i+1}_{2j_0+5} & e^{i+1}_{2j_0+7}        & o^{i+1}_{2j_0+7}\\
\hline
         {0.49} &   1.0e-04 &      &    1.6e-04 &       &   1.1e-05 &      &    2.6e-08 &      \\       {0.49} &   1.3e-07 &   9.6&    9.1e-07 &   7.5 &   4.1e-08 &   8.0&    9.9e-11 &   8.0\\
      {0.49} &   1.6e-08 &   3.0&    3.8e-09 &   7.9 &   1.7e-09 &   4.5&    3.2e-13 &   8.2\\      {0.49} &   6.8e-10 &   4.5&    8.0e-11 &   5.5 &   5.6e-11 &   4.9&    1.3e-15 &   7.9\\
\hline
\end{array}
\end{equation*}
\caption{Errors, $e^{i+1}_{(2j_0-1)+\xi}$, and orders of approximation, $o^{i+1}_{(2j_0-1)+\xi}$, with $\xi=2,4,5,8$ for the different methods, with $f_1(x), x \in [-2,2]$}
\label{tablaf1r4}
}
\end{table}

\begin{table}[H]
{\footnotesize
\begin{equation*}
\begin{array}{|c|cc|cc|cc|cc|}
\multicolumn{9}{c}{\text{Linear interpolation, 8-points} }\\
\hline
      &    e^{i+1}_{2j_0+1}        & o^{i+1}_{2j_0+1}& e^{i+1}_{2j_0+3}        & o^{i+1}_{2j_0+3}& e^{i+1}_{2j_0+5}        & o^{i+1}_{2j_0+5} & e^{i+1}_{2j_0+7}        & o^{i+1}_{2j_0+7}\\
\hline
    {0.49}                                            &  9.8e-2 &          &  1.9e-2 &          &  5.4e-3 &          &  4.5e-3 &         \\
    {0.5}  &    9.8e-2 &  1.2e-3 &  2.1e-2 & -1.0e-1 &  2.4e-3 &  1.1e+00 &  7.5e-6 &  9.2\\
     {0.5}  &    9.8e-2 & -1.6e-5 &  2.1e-2 & -1.1e-4 &  2.4e-3 &  2.7e-03 &  9.7e-9 &  9.6\\
   {0.5}  &   9.8e-2 & -1.0e-7 &  2.1e-2 &  1.0e-7 &  2.4e-3 &  2.4e-06 &  6.1e-13 &  13.9\\
\hline
\multicolumn{9}{c}{\text{Rational interpolation,} \quad r=4}\\
\hline
     &    e^{i+1}_{2j_0+1}        & o^{i+1}_{2j_0+1}& e^{i+1}_{2j_0+3}        & o^{i+1}_{2j_0+3}& e^{i+1}_{2j_0+5}        & o^{i+1}_{2j_0+5} & e^{i+1}_{2j_0+7}        & o^{i+1}_{2j_0+7}\\
\hline
    {0.55}&    1.9e-1 &     &   4.0e-2 &      &  2.8e-3 &      &  1.4915e-2 &            \\         {0.44}&    2.0e-3 &  6.6&   2.6e-3 &  3.9 &  1.2e-3 &  1.2 &  2.8851e-4 &  5.6       \\
     {0.44}&    3.2e-6 &  9.2&   1.1e-6 &  11.2&  1.7e-7 &  12.7&  8.2337e-8 &  11.7      \\        {0.49}&    2.5e-8 &  6.9&   8.8e-9 &  \textbf{6.9} &  1.2e-9 &  \textbf{7.1} &  7.2834e-11 &  10.1      \\
\hline
\multicolumn{9}{c}{\text{WENO}, r=4. }\\
\hline
    &    e^{i+1}_{2j_0+1}        & o^{i+1}_{2j_0+1}& e^{i+1}_{2j_0+3}        & o^{i+1}_{2j_0+3}& e^{i+1}_{2j_0+5}        & o^{i+1}_{2j_0+5} & e^{i+1}_{2j_0+7}        & o^{i+1}_{2j_0+7}\\
\hline
      {0.5}     &  3.5e-2   &      &   1.7e-2  &        &   2.5e-2  &        &   9.4e-2  &        \\
    {0.5}     &  7.0e-4   &  5.6 &   4.2e-4  &   5.3  &   5.0e-5  &   8.9  &   4.0e-6  &   14.5 \\           {0.5}     &  5.8e-6   &  6.9 &   3.9e-6  &   6.7  &   2.6e-7  &   7.6  &   1.0e-8  &   8.55 \\
   {0.5}     &  5.5e-8   &  6.7 &   3.3e-8  &   6.8  &   1.6e-9  &   7.2  &   5.0e-13  &   14.3 \\  \hline
\end{array}
\end{equation*}
\caption{Errors, $e^{i+1}_{(2j_0-1)+\xi}$, and orders of approximation, $o^{i+1}_{(2j_0-1)+\xi}$, with $\xi=2,4,5,8$ for the different methods, with $f_2(x), x \in [-2,2]$}
\label{tablaf2r4}
}
\end{table}

\end{document}